\documentclass[11pt]{amsart}

\usepackage[T1]{fontenc}
\usepackage[utf8]{inputenc}
\usepackage{lmodern}
\usepackage{microtype}
\usepackage{amsmath,amssymb,amsthm,mathtools}
\usepackage{enumitem}
\usepackage[colorlinks=true,linkcolor=blue,citecolor=blue,urlcolor=blue]{hyperref}

\newtheorem{theorem}{Theorem}[section]
\newtheorem{lemma}[theorem]{Lemma}

\newtheorem{corollary}[theorem]{Corollary}
\theoremstyle{definition}
\newtheorem{conjecture}[theorem]{Conjecture}
\newtheorem{remark}[theorem]{Remark}

\newcommand{\RR}{\mathcal{R}}
\newcommand{\PP}{\mathbb{P}}
\newcommand{\EE}{\mathbb{E}}
\newcommand{\mtwo}{m_2}

\DeclareMathOperator{\vtx}{v}
\DeclareMathOperator{\edg}{e}

\title[Ramsey-finite graph pairs]{Ramsey-finite graph pairs: A complete solution to the Burr-Erd\H{o}s-Faudree-Rousseau-Schelp conjectures}
\author{Yaping Mao}
\thanks{Supported by the National Natural Science Foundation of China (Nos.~12471329 and 12061059).}
\address{Academy of Plateau Science and Sustainability, and School of Mathematics and Statistics, Qinghai Normal University, Xining 810008, China}
\email{yapingmao@outlook.com; myp@qhnu.edu.cn}

\begin{document}
\begin{abstract}
For finite graphs $G$ and $H$, let $\RR(G,H)$ denote the isomorphism classes of Ramsey-minimal graphs for $(G,H)$. We prove two 1981 conjectures of Burr, Erd\H{o}s, Faudree, Rousseau, and Schelp: Ramsey-finiteness is preserved by adjoining disjoint matchings, and $(G,H)$ is Ramsey-infinite unless both graphs are odd stars or one graph has a $K_2$ component. We also replace Burr's stronger 1979 survey characterization by the correct necessary-and-sufficient form: apart from the matching case and the odd-star-with-matchings case, the only additional finite pairs are Faudree's star-forest family.\\[2mm]
\end{abstract}

\maketitle

\section{Introduction}

All graphs in this paper are finite and simple. Unless explicitly stated otherwise, all target graphs have no isolated vertices and have at least one edge. For a graph $X$, let $V(X)$ and $E(X)$ denote its vertex and edge sets, and write $\vtx(X)=|V(X)|$ and $\edg(X)=|E(X)|$. If $Y$ is isomorphic to a subgraph of $X$, we write $Y\preceq X$. Throughout, a subgraph need not be induced; when a particular subgraph is fixed, we also write $Y\subseteq X$. For an edge $e\in E(X)$, the graph $X-e$ is obtained by deleting $e$ and keeping the same vertex set. The disjoint union of two graphs is denoted by $X\cup Y$, and all graph unions in the paper are vertex-disjoint. For an integer $j\ge0$, $jK_2$ denotes the disjoint union of $j$ copies of $K_2$; thus $jK_2$ is a matching of size $j$, and $0K_2$ is the empty graph. A \emph{$K_2$-component}, also called a single-edge component, means a connected component isomorphic to $K_2$, not merely an edge subgraph. We write
\[
  S(r)=K_{1,r}
\]
for the star with $r$ edges, so that $S(1)=K_2$. A \emph{star forest} is a forest whose connected components are stars, and an \emph{odd star} is a star $S(r)$ with $r$ odd. In statements where the matching case is separated off, the odd stars are written with at least three edges to avoid duplicating the case $S(1)=K_2$.

For graphs $F$, $G$, and $H$, we write
\[
  F\to(G,H)
\]
if every red-blue coloring of the edges of $F$ contains either a red copy of $G$ or a blue copy of $H$. A graph $F$ is \emph{Ramsey-minimal} for the pair $(G,H)$ if $F\to(G,H)$ but no proper subgraph of $F$ has this property. Such an $F$ has no isolated vertices; consequently, under our standing hypotheses this is equivalent to requiring $F-e\not\to(G,H)$ for every edge $e\in E(F)$. The family of all Ramsey-minimal graphs for $(G,H)$, considered up to isomorphism, is denoted by $\RR(G,H)$. The pair $(G,H)$ is called \emph{Ramsey-finite} if $\RR(G,H)$ is finite, and \emph{Ramsey-infinite} otherwise. The Ramsey relation is symmetric in the two target graphs: by swapping the two colors, $F\to(G,H)$ if and only if $F\to(H,G)$, and hence $\RR(G,H)=\RR(H,G)$ as a set of isomorphism classes.

This finiteness problem is different from the usual Ramsey-number problem. Ramsey's theorem guarantees the existence of some complete graph $K_N$ with $K_N\to(G,H)$, but it does not describe the edge-minimal graphs with this property. The class $\RR(G,H)$ records all Ramsey-minimal witnesses to the Ramsey property. In many natural cases it is infinite, often with members of arbitrarily large order and very different structure. The Ramsey-finite cases are therefore exceptional, and the main question is to identify precisely which pairs $(G,H)$ produce only finitely many minimal witnesses.

The study of Ramsey-minimal graphs began with the work of Burr, Erd\H{o}s, and Lov\'asz \cite{BEL76}, who investigated graphs of Ramsey type and already exhibited infinite phenomena for complete graphs. Ne\v{s}et\v{r}il and R\"odl proved several foundational infinitude results: for instance, $\RR(G,H)$ is infinite when both graphs are sufficiently connected, and also in the forest setting when neither graph is a union of stars \cite{NR78a,NR78b}. These results introduced a guiding theme that persists throughout the subject: most graph pairs are Ramsey-infinite, while finite pairs arise from special low-complexity structures such as matchings and stars.

The first major finite family was obtained by Burr, Erd\H{o}s, Faudree, and Schelp, who proved that if one member of the pair is a matching, then the pair is Ramsey-finite \cite{BEFS78}. Around the same period, a series of papers of Burr, Erd\H{o}s, Faudree, Rousseau, and Schelp treated stars, star forests, and forests. For star forests with no single-edge components, they proved that Ramsey-finiteness occurs exactly when both graphs are single odd stars \cite{BERS81}. For forests more generally, they proved broad infinitude theorems whenever a non-star component is present \cite{BEFRS82}. Faudree later completed the forest case by giving an explicit classification of the Ramsey-finite pairs of forests \cite{Fa91}.

There is also a substantial body of work on Ramsey-infinite pairs outside the forest setting. Burr, Erd\H{o}s, Faudree, and Schelp proved infinitude results for pairs consisting of a star and a nontrivial connected or $2$-connected graph \cite{BEFS80}. Burr, Faudree, and Schelp developed block-based criteria for Ramsey-infiniteness \cite{BFS85}. \L{}uczak proved that if one graph is a forest that is not a matching and the other contains a cycle, then the pair is Ramsey-infinite \cite{Lu94}. Bollob\'as, Donadelli, Kohayakawa, and Schelp studied further cyclic cases, including cycle versus suitable $2$-connected graphs \cite{BDKS01}. R\"odl and Siggers later showed that for complete graphs there are exponentially many nonisomorphic Ramsey-minimal graphs of bounded order, strengthening the qualitative infinitude picture \cite{RS08}.

Ramsey-minimal graphs have also been studied from several more structural viewpoints. Examples include the classification of particular minimal families such as $\RR(K_{1,2},K_{1,m})$ and $\RR(K_{1,2},K_3)$ \cite{BHS04,BSS05}, constructions involving stars versus graph families \cite{BH12}, and descriptions of minimal graphs when one target is a matching \cite{WBAS17}. Another active direction concerns minimum degrees of Ramsey-minimal graphs. Starting from the parameter introduced by Burr, Erd\H{o}s, and Lov\'asz \cite{BEL76}, later work obtained bounds and exact results for many classes of graphs and color numbers; see, for example, \cite{SZZ10,FGLPS16,HRS18,BBL22}. These developments are not needed in the proofs below, but they place the present finiteness problem within the broader theory of Ramsey-minimal graphs.

The only recent ingredient used in this paper comes from random Ramsey theory. R\"odl and Ruci\'nski located the threshold for the symmetric random Ramsey property \cite{RR95}. The asymmetric analogue was conjectured by Kohayakawa and Kreuter \cite{KK97}. Its $1$-statement was proved by Mousset, Nenadov, and Samotij \cite{MNS20}; further threshold cases were obtained by Kuperwasser and Samotij \cite{KS24}; and the conjecture was resolved by Christoph, Martinsson, Steiner, and Wigderson \cite{CMSW25}. We use this theorem only to prove that if both $G$ and $H$ contain a cycle, then $\RR(G,H)$ is infinite. Once this cyclic--cyclic case is available, the remaining cases follow from the classical forest and forest--cyclic literature described above.

We now recall the two 1981 conjectural statements that motivate the first part of the paper. They were proposed by Burr, Erd\H{o}s, Faudree, Rousseau, and Schelp in their paper on Ramsey-minimal graphs for matchings.

\begin{conjecture}[Burr--Erd\H{o}s--Faudree--Rousseau--Schelp \cite{BEFRS81}]\label{conj:add-matchings}
If $(G,H)$ is Ramsey-finite, then $(G\cup \ell K_2, H\cup mK_2)$ is Ramsey-finite for all integers $\ell,m\ge 0$.
\end{conjecture}

\begin{conjecture}[Burr--Erd\H{o}s--Faudree--Rousseau--Schelp \cite{BEFRS81}]\label{conj:one-sided}
The pair $(G,H)$ is Ramsey-infinite unless both $G$ and $H$ are odd stars or at least one of $G$ and $H$ contains a $K_2$-component.
\end{conjecture}

It is important that Conjecture~\ref{conj:one-sided} is a one-sided statement. The original paper explicitly notes that the converse fails: there are pairs with a $K_2$-component which are still Ramsey-infinite \cite[p.~161]{BEFRS81}. Thus the conjecture does not say that the two exceptional classes are exactly the finite classes; rather, it says that outside these classes infinitude must occur.

Burr's earlier survey suggested a stronger finite characterization. The two-alternative formulation is not correct as stated, because it misses one finite star-forest family. The corrected necessary-and-sufficient form is as follows.

\begin{theorem}\label{thm:corrected-burr}
Let $G$ and $H$ be finite graphs with at least one edge and without isolated vertices. Then $\RR(G,H)$ is finite if and only if, after possibly interchanging $G$ and $H$, one of the following alternatives holds:
\begin{enumerate}[label=\textup{(\roman*)}, leftmargin=2.8em]
    \item $H=tK_2$ for some integer $t\ge 1$;
    \item $G=S(r)\cup aK_2$ and $H=S(s)\cup bK_2$ for odd integers $r,s\ge 3$ and integers $a,b\ge 0$;
    \item there are integers $a,b\ge 0$, an integer $d\ge 2$, an odd integer $r\ge 3$, and integers
    \[
        m_1\ge m_2\ge \cdots \ge m_d\ge 2,
    \]
    such that
    \[
        G=\bigcup_{i=1}^d S(m_i)\cup aK_2,\qquad
        H=S(r)\cup bK_2,
    \]
    where $m_1$ is odd, $m_1\ge r+m_2-1$, and $b\ge N_0$. Here $N_0=N_0(m_1,\ldots,m_d,r,a)$ is the threshold integer supplied by Faudree's theorem for this fixed star-forest data.
\end{enumerate}
\end{theorem}

The integer $N_0$ in Theorem~\ref{thm:corrected-burr}\textup{(iii)} is an existence threshold from Faudree's classification. It depends only on the fixed non-matching star components and on the matching size already present on the many-star side; it does not depend on the variable matching size $b$ on the one-star side. An explicit formula is not needed for the equivalence. In the concrete counterexample recalled in Section~5, the older star-forest theorem gives an explicit numerical bound.

The purpose of the present paper is to prove the two 1981 conjectures and to derive the corrected finite characterization above. More precisely, we prove the following theorem.

\begin{theorem}\label{thm:main-summary}
Let $G$ and $H$ be finite graphs with at least one edge and without isolated vertices.
\begin{enumerate}[label=\textup{(\roman*)}, leftmargin=2.6em]
    \item If neither $G$ nor $H$ has a $K_2$-component and $G,H$ are not both odd stars, then $\RR(G,H)$ is infinite.
    \item If $\RR(G,H)$ is finite, then $\RR(G\cup \ell K_2, H\cup mK_2)$ is finite for all $\ell,m\ge 0$.
    \item The complete finite classification is given by Theorem~\ref{thm:corrected-burr}.
\end{enumerate}
\end{theorem}

The first assertion is exactly Conjecture~\ref{conj:one-sided}. The second assertion is Conjecture~\ref{conj:add-matchings}. The third assertion is the corrected necessary-and-sufficient version of Burr's stronger survey formulation; the extra alternative in Theorem~\ref{thm:corrected-burr}\textup{(iii)} is precisely the finite star-forest family omitted from the original two-alternative statement.

The paper is organized as follows. Section~2 collects the classical results used in the proof. Section~3 proves the cyclic--cyclic case using the asymmetric random Ramsey threshold theorem; see Theorem \ref{thm:cyclic-cyclic}. Section~4 derives the two 1981 conjectures. Section~5 proves the corrected Burr--Faudree characterization and records Faudree's complete classification of Ramsey-finite forest pairs.

\section{Classical input}

We shall use the following known results.

\begin{theorem}[Burr--Erd\H{o}s--Faudree--Schelp \cite{BEFS78}]\label{thm:matching}
If one of $G$ and $H$ is a matching, then $\RR(G,H)$ is finite.
\end{theorem}

\begin{theorem}[Burr--Erd\H{o}s--Faudree--Rousseau--Schelp \cite{BERS81}]\label{thm:star-forests}
Let $G$ and $H$ be star forests with no $K_2$-components. Then $\RR(G,H)$ is finite if and only if both $G$ and $H$ are odd stars.
\end{theorem}

\begin{theorem}[Burr--Erd\H{o}s--Faudree--Rousseau--Schelp \cite{BEFRS82}]\label{thm:forests-nonstar}
If $G$ and $H$ are forests, neither is a matching, and at least one of them has a component that is not a star, then $\RR(G,H)$ is infinite.
\end{theorem}

\begin{theorem}[{\L}uczak \cite{Lu94}]\label{thm:luczak}
If $F$ is a forest which is not a matching and $G$ contains a cycle, then $\RR(G,F)$ is infinite.
\end{theorem}

\begin{lemma}\label{lem:faudree-remark}
Let $F_1$ and $F_2$ be forests. If $\RR(F_1,F_2)$ is finite, then $\RR(F_1,F_2\cup K_2)$ is finite. This is the remark on p.~123 of \cite{Fa91}.
\end{lemma}

By symmetry, Lemma~\ref{lem:faudree-remark} also yields finiteness of $\RR(F_1\cup K_2,F_2)$ whenever $\RR(F_1,F_2)$ is finite.

\section{The cyclic--cyclic case}

For a graph $X$, let
\[
\rho(X):=\max\Bigl\{\frac{\edg(J)}{\vtx(J)}: J\subseteq X,\ \vtx(J)\ge 1\Bigr\}.
\]
If $X$ contains a cycle, define
\[
\mtwo(X):=\max\Bigl\{\frac{\edg(J)-1}{\vtx(J)-2}: J\subseteq X,\ \vtx(J)\ge 3\Bigr\}.
\]
Whenever $\mtwo(G)\ge \mtwo(H)>1$, set
\[
\mtwo(G,H):=\max\Bigl\{\frac{\edg(J)}{\vtx(J)-2+1/\mtwo(H)}: J\subseteq G,\ \vtx(J)\ge 2\Bigr\}.
\]

We write $G_{n,p}$ for the binomial random graph on vertex set $[n]$ with edge probability $p$.
The next theorem is the special case we need from the resolved Kohayakawa--Kreuter conjecture.

\begin{theorem}[Asymmetric Ramsey threshold theorem \cite{MNS20,KS24,CMSW25}]\label{thm:threshold}
Let $G$ and $H$ be finite graphs, each containing a cycle, and assume $\mtwo(G)\ge \mtwo(H)$. Then there exist constants $0<c<C$ such that
\[
\lim_{n\to\infty}\PP\bigl(G_{n,p}\to(G,H)\bigr)=
\begin{cases}
0,& p\le c n^{-1/\mtwo(G,H)},\\[1mm]
1,& p\ge C n^{-1/\mtwo(G,H)}.
\end{cases}
\]
\end{theorem}

We are now in a position to prove the key cyclic--cyclic infinitude result. 
\begin{theorem}\label{thm:cyclic-cyclic}
If both $G$ and $H$ contain a cycle, then $\RR(G,H)$ is infinite.
\end{theorem}

\begin{proof}
Because $F\to(G,H)$ if and only if $F\to(H,G)$, we may assume $\mtwo(G)\ge \mtwo(H)$ and set
\[
d:=\mtwo(G,H).
\]
Suppose, for a contradiction, that $\RR(G,H)$ is finite. Let $F_1,\dots,F_t$ be representatives of all isomorphism classes in $\RR(G,H)$. Then for every graph $X$,
\begin{equation}\label{eq:equiv}
X\to(G,H)\quad\Longleftrightarrow\quad \text{$F_i\preceq X$ for some $i\in\{1,\dots,t\}$.}
\end{equation}
Indeed, if $X\to(G,H)$, choose a subgraph of $X$ which is minimal with respect to inclusion among all Ramsey graphs for $(G,H)$; this subgraph belongs to $\RR(G,H)$ and hence is isomorphic to one of the representatives $F_i$, that is, $F_i\preceq X$ for some $i$. The converse follows from monotonicity.

We first show that every $F_i$ has density strictly larger than $d$.

\smallskip
\noindent\emph{Claim.} For each $i\in\{1,\dots,t\}$, one has $\rho(F_i)>d$.

\smallskip
\noindent\emph{Proof of the claim.} Suppose instead that $\rho(F_i)\le d$ for some fixed $i$. Let $X_i$ count the labeled copies of $F_i$ in $G_{n,p}$ with $p=cn^{-1/d}$, where $c$ is the constant from Theorem~\ref{thm:threshold}. Since $F_i$ is fixed,
\[
\EE X_i=\Theta\!\left(n^{\vtx(F_i)}p^{\edg(F_i)}\right)=\Theta\!\left(n^{\vtx(F_i)-\edg(F_i)/d}\right).
\]
Because $\edg(F_i)/\vtx(F_i)\le \rho(F_i)\le d$, the exponent is nonnegative, so $\EE X_i$ is bounded away from $0$.

For the second moment, group ordered pairs of labeled copies by the isomorphism type of their intersection subgraph. Here the intersection is identified with a subgraph $L$ of $F_i$, and the empty intersection is allowed. This gives
\[
\EE X_i^2=\sum_{L\subseteq F_i} O\!\left(n^{2\vtx(F_i)-\vtx(L)}p^{2\edg(F_i)-\edg(L)}\right).
\]
Dividing by $(\EE X_i)^2$, we obtain
\[
\frac{\EE X_i^2}{(\EE X_i)^2}=O\!\left(\sum_{L\subseteq F_i} n^{-\vtx(L)+\edg(L)/d}\right).
\]
For every nonempty $L\subseteq F_i$, we have $\edg(L)/\vtx(L)\le \rho(F_i)\le d$, so $-\vtx(L)+\edg(L)/d\le 0$; the empty intersection contributes $1$. Hence $\EE X_i^2=O((\EE X_i)^2)$, and the second moment method yields
\[
\PP(X_i>0)\ge \frac{(\EE X_i)^2}{\EE X_i^2}\ge \eta_i>0
\]
for some constant $\eta_i$ and all sufficiently large $n$. Since $X_i>0$ implies $F_i\preceq G_{n,p}$, equivalence \eqref{eq:equiv} gives
\[
\limsup_{n\to\infty}\PP\bigl(G_{n,p}\to(G,H)\bigr)\ge \eta_i>0,
\]
contradicting the $0$-statement in Theorem~\ref{thm:threshold}. This proves the claim.
\medskip

Now fix $i$. By the claim, choose a subgraph $J_i\subseteq F_i$ such that
\[
\frac{\edg(J_i)}{\vtx(J_i)}=\rho(F_i)>d.
\]
Let $Y_i$ count the labeled copies of $J_i$ in $G_{n,p}$ with $p=Cn^{-1/d}$, where $C$ is the constant from Theorem~\ref{thm:threshold}. Then
\[
\EE Y_i=\Theta\!\left(n^{\vtx(J_i)}p^{\edg(J_i)}\right)=\Theta\!\left(n^{\vtx(J_i)-\edg(J_i)/d}\right)\to 0,
\]
because $\edg(J_i)/\vtx(J_i)=\rho(F_i)>d$. By Markov's inequality,
\[
\PP(F_i\preceq G_{n,p})\le \PP(Y_i>0)\le \EE Y_i\to 0.
\]
Summing over $i$ and using \eqref{eq:equiv}, we get
\[
\PP\bigl(G_{n,p}\to(G,H)\bigr)
\le \sum_{i=1}^t \PP(F_i\preceq G_{n,p})
\longrightarrow 0,
\]
contradicting the $1$-statement in Theorem~\ref{thm:threshold}. Hence $\RR(G,H)$ is infinite.
\end{proof}

\section{The two 1981 conjectures}

We now derive the original one-sided conjecture.

\begin{theorem}\label{thm:one-sided}
Let $G$ and $H$ be finite graphs with at least one edge and without isolated vertices. Assume that neither $G$ nor $H$ has a $K_2$-component. If $G$ and $H$ are not both odd stars, then $\RR(G,H)$ is infinite.
\end{theorem}

\begin{proof}
We distinguish three cases.

\smallskip
\noindent\emph{Case 1: both $G$ and $H$ are forests.}
If both are star forests, then Theorem~\ref{thm:star-forests} applies, because neither graph has a $K_2$-component; under this hypothesis an odd star has at least three edges. Thus $\RR(G,H)$ is infinite unless both graphs are odd stars. If at least one of $G$ and $H$ has a component that is not a star, then neither graph is a matching and Theorem~\ref{thm:forests-nonstar} implies that $\RR(G,H)$ is infinite.

\smallskip
\noindent\emph{Case 2: exactly one of $G$ and $H$ is a forest.}
Without loss of generality, let $H$ be a forest and let $G$ contain a cycle. Since $H$ has no $K_2$-component, it is not a matching. Therefore Theorem~\ref{thm:luczak} yields that $\RR(G,H)$ is infinite.

\smallskip
\noindent\emph{Case 3: both $G$ and $H$ contain a cycle.}
This is exactly Theorem~\ref{thm:cyclic-cyclic}.
\end{proof}

The following consequence is convenient for the proof of Conjecture~\ref{conj:add-matchings}.

\begin{corollary}\label{cor:finite-forest}
If $\RR(G,H)$ is finite and neither $G$ nor $H$ is a matching, then both $G$ and $H$ are forests.
\end{corollary}

\begin{proof}
Suppose one of the graphs, say $G$, contains a cycle. If $H$ also contains a cycle, then Theorem~\ref{thm:cyclic-cyclic} gives a contradiction. If $H$ is a forest, then $H$ is not a matching by assumption, and Theorem~\ref{thm:luczak} again gives a contradiction. Therefore neither graph contains a cycle.
\end{proof}

We can now prove the matching-extension conjecture.

\begin{theorem}\label{thm:add-matchings}
Let $G$ and $H$ be finite graphs with at least one edge and without isolated vertices. If $\RR(G,H)$ is finite, then $\RR(G\cup \ell K_2, H\cup mK_2)$ is finite for all integers $\ell,m\ge 0$.
\end{theorem}

\begin{proof}
Assume that $\RR(G,H)$ is finite.

If one of $G$ and $H$ is a matching, say $G=tK_2$, then $G\cup \ell K_2=(t+\ell)K_2$ is again a matching. Hence Theorem~\ref{thm:matching} immediately gives that $\RR(G\cup \ell K_2, H\cup mK_2)$ is finite.

Assume now that neither $G$ nor $H$ is a matching. By Corollary~\ref{cor:finite-forest}, both $G$ and $H$ are forests. Starting from $\RR(G,H)$ finite, Lemma~\ref{lem:faudree-remark} implies that $\RR(G,H\cup K_2)$ is finite. Iterating, we obtain finiteness of $\RR(G,H\cup mK_2)$ for every $m\ge 0$. By symmetry, $\RR(H\cup mK_2,G)$ is also finite. Applying Lemma~\ref{lem:faudree-remark} another $\ell$ times to the second coordinate gives finiteness of $\RR(H\cup mK_2,G\cup \ell K_2)$, and symmetry once more yields finiteness of $\RR(G\cup \ell K_2,H\cup mK_2)$.
\end{proof}

Theorem~\ref{thm:one-sided} is precisely Conjecture~\ref{conj:one-sided}. Thus the original 1981 one-sided conjecture is true.

\section{The corrected Burr--Faudree characterization}

We first record Faudree's forest classification in a form that separates the matching case from the genuinely star-forest cases.

\begin{theorem}[\cite{Fa91}]\label{cor:faudree}
Let $F_1$ and $F_2$ be forests without isolated vertices. Then $\RR(F_1,F_2)$ is finite if and only if, after possibly interchanging $F_1$ and $F_2$, one of the following alternatives holds:
\begin{enumerate}[label=\textup{(\roman*)}, leftmargin=2.8em]
    \item $F_2=nK_2$ for some integer $n>0$;
    \item $F_1=S(m_1)\cup mK_2$ and $F_2=S(n_1)\cup nK_2$, where $m,n\ge 0$ and $m_1,n_1\ge 3$ are odd;
    \item for some integers $d\ge 2$ and $m,n\ge 0$,
    \[
        F_1=\bigcup_{i=1}^d S(m_i)\cup mK_2,
        \qquad
        F_2=S(r)\cup nK_2,
    \]
    where
    \[
        m_1\ge \cdots \ge m_d\ge 2,
    \]
    $r\ge 3$, both $m_1$ and $r$ are odd, $m_1\ge r+m_2-1$, and $n\ge N_0$. Here $N_0=N_0(m_1,\ldots,m_d,r,m)$ is the threshold integer supplied by Faudree's theorem for this fixed star-forest data; in particular, $N_0$ does not depend on $n$.
\end{enumerate}
\end{theorem}

\begin{proof}[Proof of Theorem~\ref{thm:corrected-burr}]
First suppose that $\RR(G,H)$ is finite. If one of $G$ and $H$ is a matching, then, after interchanging the two graphs if necessary, alternative~\textup{(i)} holds. Assume then that neither graph is a matching. By Corollary~\ref{cor:finite-forest}, both $G$ and $H$ are forests. Applying Theorem~\ref{cor:faudree}, and excluding its matching alternative, gives alternatives~\textup{(ii)} and \textup{(iii)} of the present theorem after the harmless renaming of parameters used there.

Conversely, if alternative~\textup{(i)} holds, then $\RR(G,H)$ is finite by Theorem~\ref{thm:matching}. If alternative~\textup{(ii)} holds, then the pair of odd stars is Ramsey-finite by Theorem~\ref{thm:star-forests}, and adjoining the displayed matchings preserves finiteness by Theorem~\ref{thm:add-matchings}. Finally, alternative~\textup{(iii)} is finite by Theorem~\ref{cor:faudree}. This proves the characterization.
\end{proof}

\begin{remark}
Burr's original survey formulation corresponds to alternatives~\textup{(i)} and \textup{(ii)} of Theorem~\ref{thm:corrected-burr}. Alternative~\textup{(iii)} is genuinely needed. For example, Theorem~11 of \cite{BERS81} states that if $\ell$, $n$, and $s$ are positive integers with $\ell$ and $n$ odd and
\[
    n\ge \ell+s-1,
\]
then
\[
    \RR\bigl(S(n)\cup S(s),\, S(\ell)\cup kK_2\bigr)
\]
is finite for every
\[
    k\ge (n+2\ell+s-2)^2+1.
\]
Taking $n=5$, $\ell=3$, $s=2$, and $k=122$ shows that
\[
    \RR\bigl(S(5)\cup S(2),\, S(3)\cup 122K_2\bigr)
\]
is finite. The corresponding Ramsey-finite pair is not covered by the two alternatives in Burr's original statement but is covered by Theorem~\ref{thm:corrected-burr}\textup{(iii)}.
\end{remark}

Combining Theorem~\ref{thm:one-sided}, Theorem~\ref{thm:add-matchings}, and Theorem~\ref{thm:corrected-burr}, we obtain Theorem~\ref{thm:main-summary} together with the complete finite classification.


\begin{thebibliography}{99}

\bibitem{BBL22}
J.~Bamberg, A.~Bishnoi, and T.~Lesgourgues,
\emph{The minimum degree of Ramsey-minimal graphs for cliques},
Bull. London Math. Soc. \textbf{54} (2022), 1827--1838.

\bibitem{BDKS01}
B.~Bollob\'as, J.~Donadelli, Y.~Kohayakawa, and R.~H. Schelp,
\emph{Ramsey minimal graphs},
J. Braz. Comput. Soc. \textbf{7} (2001), 27--37.

\bibitem{BEL76}
S.~A. Burr, P.~Erd\H{o}s, and L.~Lov\'asz,
\emph{On graphs of Ramsey type},
Ars Combin. \textbf{1} (1976), 167--190.

\bibitem{BFS85}
S.~A. Burr, R.~J. Faudree, and R.~H. Schelp,
\emph{On graphs with Ramsey-infinite blocks},
European J. Combin. \textbf{6} (1985), 129--132.

\bibitem{BHS04}
M.~Borowiecki, M.~Ha\l{}uszczak, and E.~Sidorowicz,
\emph{On Ramsey minimal graphs},
Discrete Math. \textbf{286} (2004), 37--43.

\bibitem{BH12}
M.~Borowiecka-Olszewska and M.~Ha\l{}uszczak,
\emph{On Ramsey $(K_{1,m},\mathcal G)$-minimal graphs},
Discrete Math. \textbf{313} (2013), 1843--1855.

\bibitem{BSS05}
M.~Borowiecki, I.~Schiermeyer, and E.~Sidorowicz,
\emph{Ramsey $(K_{1,2},K_3)$-minimal graphs},
Electron. J. Combin. \textbf{12} (2005), Paper R20, 15 pp.


\bibitem{BEFS78}
S.~A. Burr, P.~Erd\H{o}s, R.~J. Faudree, and R.~H. Schelp,
\emph{A class of Ramsey-finite graphs},
Proc. 9th Southeastern Conf. on Combinatorics, Graph Theory and Computing,
1978, pp.~171--180.


\bibitem{BEFS80}
S.~A. Burr, P.~Erd\H{o}s, R.~J. Faudree, and R.~H. Schelp,
\emph{Ramsey minimal graphs for the pair star, connected graph},
Studia Sci. Math. Hungar. \textbf{15} (1980), 265--273.

\bibitem{BEFRS81}
S.~A. Burr, P.~Erd\H{o}s, R.~J. Faudree, C.~C. Rousseau, and R.~H. Schelp,
\emph{Ramsey-minimal graphs for matchings},
in \emph{The Theory and Applications of Graphs} (Kalamazoo, Mich., 1980),
Wiley, New York, 1981, pp.~159--168.

\bibitem{BEFRS82}
S.~A. Burr, P.~Erd\H{o}s, R.~J. Faudree, C.~C. Rousseau, and R.~H. Schelp,
\emph{Ramsey-minimal graphs for forests},
Discrete Math. \textbf{38} (1982), 23--32.

\bibitem{BERS81}
S.~A. Burr, P.~Erd\H{o}s, R.~J. Faudree, C.~C. Rousseau, and R.~H. Schelp,
\emph{Ramsey-minimal graphs for star-forests},
Discrete Math. \textbf{33} (1981), 227--237.

\bibitem{Bu79}
S.~A. Burr,
\emph{A survey of noncomplete Ramsey theory for graphs},
in \emph{Topics in Graph Theory},
Ann. New York Acad. Sci. \textbf{328} (1979), 58--75.

\bibitem{CMSW25}
M.~Christoph, A.~Martinsson, R.~Steiner, and Y.~Wigderson,
\emph{Resolution of the Kohayakawa--Kreuter conjecture},
Proc. Lond. Math. Soc. \textbf{130} (2025), e70013.

\bibitem{Fa91}
R.~Faudree,
\emph{Ramsey minimal graphs for forests},
Ars Combin. \textbf{31} (1991), 117--124.

\bibitem{Lu94}
T.~{\L}uczak,
\emph{On Ramsey minimal graphs},
Electron. J. Combin. \textbf{1} (1994), Research Paper~4, 4~pp.


\bibitem{FGLPS16}
J.~Fox, A.~Grinshpun, A.~Liebenau, Y.~Person, and T.~Szab\'o,
\emph{On the minimum degree of Ramsey-minimal graphs for multiple colours},
J. Combin. Theory Ser. B \textbf{120} (2016), 64--82.

\bibitem{HRS18}
H.~H\`an, V.~R\"odl, and T.~Szab\'o,
\emph{Vertex Folkman numbers and the minimum degree of Ramsey-minimal graphs},
SIAM J. Discrete Math. \textbf{32} (2018), 826--838.

\bibitem{KK97}
Y.~Kohayakawa and B.~Kreuter,
\emph{Threshold functions for asymmetric Ramsey properties involving cycles},
Random Structures Algorithms \textbf{11} (1997), 245--276.

\bibitem{KS24}
E.~Kuperwasser and W.~Samotij,
\emph{The list-Ramsey threshold for families of graphs},
Combin. Probab. Comput. \textbf{33} (2024), 829--851.

\bibitem{MNS20}
F.~Mousset, R.~Nenadov, and W.~Samotij,
\emph{Towards the Kohayakawa--Kreuter conjecture on asymmetric Ramsey properties},
Combin. Probab. Comput. \textbf{29} (2020), 943--955.

\bibitem{NR78a}
J.~Ne\v{s}et\v{r}il and V.~R\"odl,
\emph{On Ramsey minimal graphs},
Colloq. Internationaux C.N.R.S. \textbf{260} (1978), 307--308.

\bibitem{NR78b}
J.~Ne\v{s}et\v{r}il and V.~R\"odl,
\emph{The structure of critical Ramsey graphs},
Acta Math. Acad. Sci. Hungar. \textbf{32} (1978), 295--300.

\bibitem{RR95}
V.~R\"odl and A.~Ruci\'nski,
\emph{Threshold functions for Ramsey properties},
J. Amer. Math. Soc. \textbf{8} (1995), 917--942.

\bibitem{RS08}
V.~R\"odl and M.~Siggers,
\emph{On Ramsey minimal graphs},
SIAM J. Discrete Math. \textbf{22} (2008), 467--488.

\bibitem{SZZ10}
T.~Szab\'o, P.~Zumstein, and S.~Z\"urcher,
\emph{On the minimum degree of Ramsey-minimal graphs},
J. Graph Theory \textbf{64} (2010), 150--164.

\bibitem{WBAS17}
K.~Wijaya, E.~T. Baskoro, H.~Assiyatun, and D.~Suprijanto,
\emph{On Ramsey $(mK_2,H)$-minimal graphs},
Graphs Combin. \textbf{33} (2017), 233--243.

\end{thebibliography}
\end{document}